\documentclass[12pt,reqno]{amsart}
\usepackage{amssymb,mathrsfs,graphicx,xcolor}
\usepackage{mathtools}
\usepackage[margin=3cm]{geometry}
\usepackage{hyperref}
\usepackage{nicefrac}

\title[The Euler-Monge-Amp\`ere system]{
  Critical threshold for global regularity of Euler-Monge-Amp\`ere system with
  radial symmetry}

\author[Eitan Tadmor]{Eitan Tadmor}
\address[Eitan Tadmor]{\newline Department of Mathematics  
 and Institute for Physical Science \& Technology, \ University of Maryland, College Park, MD 20742, USA}
\email{tadmor@umd.edu}

\author[Changhui Tan]{Changhui Tan}
\address[Changhui Tan]{\newline Department of Mathematics, \ 
 University of South Carolina, Columbia, SC 29208, USA}
\email{tan@math.sc.edu}

\subjclass[2010]{35Q35, 35B30, 35K96, 76N10.}

\keywords{Eulerian-Monge-Amp\`ere system, critical threshold, radial
symmetry}

\newtheorem{theorem}{Theorem}[section]
\newtheorem{lemma}[theorem]{Lemma}
\newtheorem{corollary}[theorem]{Corollary}
\newtheorem{proposition}[theorem]{Proposition}
\newtheorem{remark}[theorem]{Remark}
\newtheorem{definition}{Definition}[section]

\renewcommand{\geq}{\geqslant}
\renewcommand{\ge}{\geqslant}
\renewcommand{\leq}{\leqslant}

\def\R{\mathbb{R}}
\def\pa{\partial}
\def\bG{\mathbf{\Gamma}}
\def\F{\mathbf{F}}
\def\f{f}

\def\U{\mathbf{U}}
\def\u{\mathbf{u}}
\def\v{\mathbf{v}}
\def\x{\mathbf{x}}
\def\X{\mathbf{X}}

\def\grad{\nabla}
\def\div{\grad\cdot}

\def\tr{\text{trace}}

\def\Id{\mathbb{I}}
\def\tilrho{\rho} 

\def\d{\textnormal{d}}
\numberwithin{equation}{section}
\begin{document}
\allowdisplaybreaks

\thanks{\textbf{Acknowledgment.} Research  was supported in part  by ONR grant N00014-2112773 (ET) and by NSF grants DMS 18-53001 and 21-08264 (CT)}
\date{\today}

\begin{abstract}
 We study the global wellposedness of the Euler-Monge-Amp\`ere (EMA)
 system. We obtain a sharp, explicit critical threshold in the space of initial configurations which guarantees the global regularity of EMA
 system with radially symmetric initial data. The result is obtained using two independent approaches --- one using \emph{spectral dynamics} of Liu \& Tadmor \cite{liu2002spectral} and another based on the \emph{geometric approach}  of Brenier \& Loeper \cite{brenier2004geometric}. The results are extended to 2D radial EMA with swirl.
\end{abstract}

\maketitle 

\setcounter{tocdepth}{1}
{\small\tableofcontents}

\section{Introduction}\label{sec:intro}
We are concerned with the global regularity of the pressureless Euler-Monge-Amp\`ere (EMA) system
\begin{subequations}\label{eqs:EMA}
\begin{align}
 \pa_t\rho+\div(\rho\u)&=0,\label{eq:density}\\
 \pa_t(\rho\u)+\div(\rho\u\otimes\u)&=-\kappa\rho\grad\phi,\label{eq:momentum}\\
 det(\Id-D^2\phi)&=\rho\label{eq:MA}
\end{align}
\end{subequations}
with density $\rho(\cdot,t): {\mathbb R}^n \mapsto {\mathbb R}_+$,
velocity $\u(\cdot,t): {\mathbb R}^n \mapsto {\mathbb R}^n$ and
potential $\phi(\cdot,t): {\mathbb R}^n \mapsto {\mathbb R}$, subject
to the corresponding initial conditions
$(\rho_0(\cdot),\u_0(\cdot),\phi_0(\cdot))$ at $t=0$.
We set the constant $\kappa>0$, representing a repulsive
force. Without loss of generality, we fix the potential assuming $\phi(0)=0$.

The EMA system \eqref{eqs:EMA} has been introduced and studied by Loeper in
\cite{loeper2005quasi}, around its equilibrium state
 $({\rho},\u)=(1,{\mathbf 0})$. It is closely related to the 
Euler-Poisson equations in plasma physics
\begin{subequations}\label{eqs:EP}
\begin{align}\label{eq:EP}
 \pa_t{\tilrho}+\div({\tilrho}\u)&=0,\\
 \pa_t({\tilrho}\u)+\div({\tilrho}\u\otimes\u)&=-\kappa{\tilrho}\,\grad\phi,\\
 -\Delta \phi&={\tilrho}-1.\label{eq:Poisson}
\end{align}
\end{subequations}
Indeed, these two systems are the same when $n=1$. In higher
dimensions, $n\geq2$, one considers a  perturbed solution  around the equilibrium
state $(1,{\mathbf 0})$: expressing $\phi=\epsilon\varphi$, then
\[
\rho=det(\Id-\epsilon D^2\varphi)= 1-\epsilon \Delta \varphi +
{\mathcal O}(\epsilon^2),
\]
which yields \eqref{eq:Poisson} modulo ${\mathcal O}(\epsilon^2)$ terms.
Hence, we can view the EMA system \eqref{eqs:EMA} as a nonlinear
counterpart of the Euler-Poisson equations \eqref{eqs:EP} around the equilibrium state.
Interestingly, if we scale $\kappa=\epsilon^{-2}$, both systems
converges to the incompressible Euler equations as $\epsilon\to0$, e.g.,  \cite{donatelli2007quasineutral,donatelli2014quasineutral}.

The stability of Euler-Poisson equations near the equilibrium state
$(\rho,\u)=(1,{\mathbf 0})$ was analyzed in
\cite{guo1998smooth,guo2011global,ionescu2013euler,guo2017absence}.
The question of  global regularity holds for a larger region in the space of initial configurations: \emph{sub-critical}  initial data admits global strong solutions while \emph{super-critical} initial
data lead to finite time singularity formations. This is known as the
\emph{critical threshold phenomenon}.
Threshold conditions for Euler-Poisson equations were found
in \cite{engelberg2001critical,liu2002spectral,tadmor2008global}
for the one-dimensional cases, and in
\cite{wei2012critical,jang2012two,tan2021eulerian}
for the multi-dimensional cases with radial symmetry.
The search for multi-dimensional threshold beyond the radial case was addressed in related \emph{restricted} models
\cite{tadmor2003critical,lee2013thresholds,lee2017upper}, but the
general unrestricted case for \eqref{eqs:EP} is left open.

In this paper we study the global regularity of the EMA system
\eqref{eqs:EMA} with radial symmetry, subject to sub-critical initial
data. We obtain a \emph{sharp} and \emph{explicit} critical threshold
in the space of initial configurations,  which distinguish between initial data admitting globally regular solutions  vs. solutions with finite-time
blowup. We state our main result.

\begin{theorem}\label{thm:main}
  Consider the EMA system \eqref{eqs:EMA} with smooth radial initial data of the form
  \begin{equation}\label{eq:radialinit}
    \rho_0(\x)=\rho_0(r),\quad \u_0(\x)=\frac{\x}{r}u_0(r),\quad
    \phi_0(\x)=\phi_0(r),\quad r=|\x|.
  \end{equation}
  Specifically, our smoothness assumption requires
  \[
  \U_0(|\x|) := \left[u_0'(|\x|), \,\,\frac{u_0(|\x|)}{|\x|}, \,\,\phi_0''(|\x|),\,\,
    \frac{\phi_0'(|\x|)}{|\x|} \right]^\top\in \big(H^s(\R^n)\big)^4, \qquad s> \frac{n}{2}.
    \]
  Then,
  \begin{itemize}
 \item Subcritical threshold: if the initial condition satisfies
   \begin{equation}\label{eq:CTcond}
     |u_0'(r)|<\sqrt{\kappa(1-2\phi_0''(r))},\qquad \mbox{for all }  r>0,
   \end{equation}
   then the system admits a global smooth solution
   \[
   \U=\left[u'(|\x|,t), \,\,\frac{u(|\x|,t)}{|\x|}, \,\,\phi''(|\x|,t),\,\,
    \frac{\phi'(|\x|,t)}{|\x|} \right]^\top\in
   C([0,T], (H^s(\R^n))^4),
   \]
   for any finite time $T$.
 \item Supercritical threshold: if  \eqref{eq:CTcond} fails to hold for some $r>0$
    then system \eqref{eqs:EMA} admits a solution which will generate singular shocks (and/or
   non-physical shocks) in a
   finite time, namely --- there exist a finite critical time $T_c$ and and a
     location $r_c=r(T_c; r_0)$ such that the solution remains smooth in $[0, T_c)$, and
     \begin{equation}\label{eq:singularshock}
       \lim_{t\to T_c-}\pa_ru(r_c,t)=-\infty,\quad \lim_{t\to
         T_c-}\rho(r_c,t)=+\infty~ ( or \ 0 ).
     \end{equation}
 \end{itemize}
\end{theorem}

\noindent
A couple of remarks is in order.
\begin{remark}[{\bf Uniqueness}]
We note that the uniqueness of our radial solution, $U(|\x|,t)$, is dictated by its vanishing behavior at infinity. In particular, the $H^s$-boundedness of $\phi_{rr}(r,t)$ and $\phi_r(r,t)/r$ imply their vanishing behavior at infinity, hence $\phi$ is dictated up to a constant by the   Monge-Amp\`ere equation (see its radial version in \eqref{eq:rhonurelation} below), which we fixed by setting, say, $\phi(0)=0$. Furthermore,  \eqref{eq:rhonurelation} then implies that $(\rho,\u)$ approaches the equilibrium state $(1,{\mathbf 0})$ at infinity.
\end{remark}

\begin{remark}[{\bf Bounded away from vacuum}]  We observe, consult  Remark \ref{rem:rhom} below, that for the $n$-dimensional  EMA threshold   \eqref{eq:CTcond} to hold, necessitates the  lower bound $\rho_0(r)>2^{-n}$. This is in agreement with the sharp 1D threshold $|u'_0(r)|<\sqrt{\kappa(2\rho_0(r)-1)}$ which requires the lower-bound $\rho_0>1/2$ --- one cannot expect a global smooth solution with initial density which is  `way below'  the equilibrium state $\rho_0\equiv 1$. In particular, therefore, the presence of
vacuum in the initial data will always lead to shock formations.
On the other hand, if $\rho_0$ is not far below from the constant equilibrium
state $\rho_0\equiv1$, then one can find subcritical initial data with
$|u_0'|$ small enough, such that the solutions exist globally in time.
\end{remark}

The proof of theorem \ref{thm:main} begins in section 
\ref{sec:radial} with a general framework established in
\cite{tan2021eulerian} on Eulerian dynamics with radial symmetry, followed by the spectral dynamics of radial EMA  in Section
\ref{sec:threshold}, and in Section \ref{sec:GWP} we complete the proof 
 of Theorem \ref{thm:main} by energy estimates. 

When the dimension $n=1$, the Monge-Amp\`ere equation \eqref{eq:MA} is simply
 $1-\phi''=\rho$. The sub-critical global regularity condition
 \eqref{eq:CTcond} is reduced to $|u'_0(r)|<\sqrt{\kappa(2\rho_0(r)-1)}$,
and our result recovers the
sharp threshold for 1D Euler-Poisson equations obtained in
\cite{engelberg2001critical}. The interesting part of theorem \ref{thm:main} comes in higher dimensions, $n\geq2$,
addressing the fully nonlinear  Monge-Amp\`ere part  of the  EMA system, which  seems  more difficult to treat when compared with the linear Poisson part in Euler-Poisson equations \eqref{eqs:EP}. 
Nevertheless,  sharp  threshold conditions for the radially symmetric Euler-Poisson equations  --- consult \cite{wei2012critical} for $u_0>0$ and 
\cite{tan2021eulerian} for general radial data, 
are stated implicitly and seem to depend on the dimension.
This is in contrast  to the explicit form of our threshold condition for the fully nonlinear EMA system,  \eqref{eq:CTcond}, which  is independent of the dimension $n$.\newline
Another perspective for having such a simple elegant threshold condition is due
to the geometric structure of the Monge-Amp\`ere equation.
In Section \ref{sec:geometric}, we pursue  the geometric approach
for the  Monge-Amp\`ere equation a la \cite{brenier2004geometric, loeper2005quasi}, and we re-derive the radial  threshold condition \eqref{eq:CTcond}.

In Section \ref{sec:nonradial},
we discuss  extension of these results  beyond radial configurations.
First, we further extend our result to the 2D radial EMA system with swirl.
\begin{theorem}\label{thm:2D}
  Consider the two-dimensional EMA system \eqref{eqs:EMA} with smooth radial initial
  data with swirl
  \begin{equation}\label{eq:radialinits}
    \rho_0(\x)=\rho_0(r),\quad \u_0(\x)=\frac{\x}{r}u_0(r)+\frac{\x^\perp}{r}\Theta_0(r),\quad
    \phi_0(\x)=\phi_0(r),\quad r=|\x|.
  \end{equation}
  Then, there exists a set $\Sigma\subset\R^6$, defined in \eqref{def:Sigma}, such that
  \begin{itemize}
 \item Subcritical threshold: if the initial condition satisfies
   \begin{equation}\label{eq:2DCTcond}
    \Big( u_0'(r),\frac{u_0(r)}{r}, \Theta_0'(r),
    \frac{\Theta_0(r)}{r}, \phi_0''(r), \frac{\phi_0'(r)}{r}\Big)\in\Sigma
   \end{equation}
   for all $r>0$, then the system admits a global smooth solution.
 \item Supercritical threshold: if there exists an $r>0$ such that \eqref{eq:2DCTcond} is
   violated, then the solution will become singular in a finite time.
 \end{itemize}
\end{theorem}

Note that the set $\Sigma$ is implicitly defined. It is not clear
whether the threshold condition \eqref{eq:2DCTcond} can be expressed 
explicitly. This indicates that rotation adds another layer of
intrinsic difficulty in extending our theory to general data.\newline
Finally, we comment on the difficulties  in both --- the approach based on spectral dynamics and the geometric approach, for extending these results to the case 
of general data.

\section{Preliminaries: Eulerian dynamics with radial symmetry}\label{sec:radial}
The EMA system \eqref{eqs:EMA} falls into a general
framework of pressureless Eulerian dynamics
\begin{align*}
& \pa_t\rho+\div(\rho\u)=0,\\
& \pa_t(\rho\u)+\div(\rho\u\otimes\u)=\rho\F,
\end{align*}
where the force $\F=-\kappa\grad\phi$ and the potential $\phi$
satisfies the Monge-Amp\`ere equation \eqref{eq:MA}.
The momentum equation can be equivalently written as the
following dynamics of the velocity $\u$, in the non-vacuous region
\begin{equation}\label{eq:velo}
\pa_t\u+(\u\cdot\grad)\u=\F.
\end{equation}
The equation the usual Eulerian-type nonlinearity and it is well-known that
the uniform boundedness of the $n\times n$ velocity gradient matrix, $\|\grad\u(\cdot,t)\|_{L^\infty}< \infty$, is the key for   global regularity.
Taking the spatial gradient of \eqref{eq:velo} would yield
\begin{equation}\label{eq:gradu}
  (\pa_t+\u\cdot\grad)\grad\u+(\grad\u)^{2}=\grad\F.
\end{equation}

\subsection{Spectral dynamics}
Let $\lambda_i=\lambda_i(\grad\u),  i=1,2,\ldots, n$ be the $n$ eigenvalues of $\grad\u$. Then, the
spectral dynamics \eqref{eq:gradu} can be written as
\begin{equation}\label{eq:lambF}
  (\pa_t+\u\cdot\grad)\lambda_i+\lambda_i^2=\langle {\mathbf l}_i, (\grad\F) {\mathbf r}_i\rangle\quad i=1,\cdots,n,
\end{equation}
where $({\mathbf l}_i, {\mathbf r}_i)$ are the corresponding left and right eigenvectors
of $\lambda_i$.

It has been studied extensively in \cite{liu2002spectral}.
Although one can largely benefit from the
explicit Ricatti structure, \eqref{eq:lambF}i, its delicate step is to control
the term on the right, $\langle {\mathbf l}_i, (\grad\F) {\mathbf r}_i\rangle$, since  in many cases $\grad\F$ need not share the same eigen-system with $\grad\u$; instead, one seeks an invariant expressed in terms of these eigen-system.
As a typical example, one study the dynamics of \emph{the divergence}
\[
d:=\div\u=\tr(\grad\u)=\sum_{i=1}^n\lambda_i(\grad\u).
\]
Taking the trace of \eqref{eq:lambF} would yield
\[
(\pa_t+\u\cdot\grad)d +\tr\left((\grad\u)^{2}\right)=\div\F.
\]
While the forcing term $\grad\cdot\F$ can be easier to control,
one loses the explicit Ricatti structure encoded in the difference,
$\tr\left((\grad\u)^{2}\right)- d^2\neq 0$ for $n\geq2$. 
The difference is related to the \emph{spectral gap},
$\lambda_1(\grad\u)-\lambda_2(\grad\u)$ (particularly in 2D).
Examples are found in \cite{tadmor2003critical,tadmor2014critical,he2017global}.

\subsection{Radially symmetric solutions}
We focus on a special type of solutions for the EMA system
\eqref{eqs:EMA}, with radial symmetry and without swirl
\begin{equation}\label{eq:radial}
  \rho(\x,t)=\rho(r,t),\quad\u(\x,t)=\frac{\x}{r}u(r,t), \quad\phi(\x,t)= \phi(r,t).
\end{equation}
Here, $r=|\x|\in\R_+$ is the radial variable, and  $\rho$, $u$ and $\phi$ are scalar
functions defined in $\R_+\times\R_+$. 
To ensure regularity at the origin, we impose boundary conditions at $r=0$
\begin{equation}\label{eq:radialzero}
  \pa_r\rho(0,t)=0,\quad u(0,t)=0,\quad \pa_r\phi(0,t)=0.
\end{equation}
The persistence of such no swirl solutions follows by noting that the velocity field is induced by radial potential 
\[
  \u(\x,t)=\grad U(\x,t), \quad U(\x,t):=\int_0^{|\x|} u(s,t){\d}s
\]
in which case \eqref{eq:velo} with $\F=-\kappa\grad\phi$ is encoded as an Eikonal equation
\begin{equation}\label{eq:eikonal}
\pa_tU +\frac{1}{2}\left|\grad U\right|^2=-\kappa\phi.
\end{equation}
The gradient of \eqref{eq:eikonal} yields the momentum equation \eqref{eq:momentum}. Take the Hessian of \eqref{eq:eikonal} to recover the velocity gradient equation 
\eqref{eq:gradu}. The next lemma  which is  at heart of matter, recalls that  radial Hessians are rank-one modifications of scalar matrix,  and hence they all share the same eigenvectors, e.g. \cite[(1.3)]{shu2021anticipation}.

\ifx
\begin{lemma}\label{lem:eigenvalue}
  Let $\f : \mathbb{R}^n\to\mathbb{R}^n$ be a radial vector
  field  $\f(\x)=\frac{\x}{r}f(r)$ where $r=|\x|$. Then, for any
  $\x\neq {\bf 0}$, the eigen-system of   $\grad\f(\x)$ is given by
  \begin{itemize}
  \item $\lambda_1=f'(r)$ and $\v_1=\x$;
  \item $\lambda_2=\cdots=\lambda_n=\frac{f(r)}{r}$, and $span\{\v_2,\cdots, \v_n\} = \{\v \,|\,\x^\top\v=0\}$.
  \end{itemize}
\end{lemma}
\fi

\begin{lemma}\label{lem:eigenvalue}
  Let $\f : \mathbb{R}^n\to\mathbb{R}$ be a radial scalar
  field  $\f(\x)=\f(|\x|)$. Then, for any
  $\x\neq {\bf 0}$, the eigen-system of   its Hessian $D^2\f(\x)$ is characterized by two distinct eigenvalues given by

\smallskip\noindent 
  $\bullet \ \displaystyle \lambda_1(D^2\f)=\f'(|\x|)$ associated with the eigenvector $\v_1=\x$;\newline
  $\bullet \ \displaystyle \lambda_2(D^2\f)=\cdots=\lambda_n(D^2\f)=\frac{\f(|\x|)}{|\x|}$, associated with
   $\{\v_2,\cdots, \v_n\}$ which span $\{\x^\perp\}$.
\end{lemma}
Indeed, the Hessian $D^2\f$ is given by 
  \[
  D^2\f(\x)=\frac{\f(r)}{r}\,\Id+\frac{1}{r^2}\left(\f'(r)-\frac{\f(r)}{r}\right)\x\,\x^\top, \qquad r:=|\x|.
  \]
  A straightforward computation yields
  \[
  \big(\lambda_1\Id-D^2\f(\x)\big)\v_1=\left(\f'(r)-\frac{\f(r)}{r}\right)
    \left(\x-\frac{1}{r^2}\x\,\langle \x,\x\right\rangle)=0,
    \]
  and for any $\v$ such that $\langle \x,\v\rangle=0$,
  \[
  \big(\lambda_2\Id-\grad\f(\x)\big)\v=-\frac{1}{r^2}\left(\f'(r)-\frac{\f(r)}{r}\right)
    \x\,\langle \x,\v\rangle=0.
    \]

It follows that the Hessians of all radial scalar fields share the same eigen-system.
In particular,  the velocity gradient, 
$\displaystyle \grad\u(\x,t)=D^2 U(r)$, and 
forcing gradient, $\grad\F(\x,t)=-\kappa D^2\phi(r)$,
  share the same eigenvectors.
Hence, we can `diagonalize'  the spectral dynamics \eqref{eq:lambF} in terms of the distinct eigenvalues of $D^2U(r)$ and of $D^2\phi(r)$, independent of the corresponding eigenvectors 
\begin{equation}\label{eq:pq}
   \begin{cases}
    p'=-p^2-\kappa \mu,\\
    q'=-q^2-\kappa \nu.
  \end{cases}
\end{equation}
Here, $':=\pa_t+u(r,t)\pa_r$ denotes differentiation along particle paths,
$(p,q)$ denote the two distinct eigenvalues of the velocity gradient $\grad\u(r,t)$ 
\begin{equation}\label{def:pq}
  p(r,t) = \lambda_1(\grad\u(\x,t))=\pa_ru(r,t),\quad
  q(r,t) = \lambda_2(\grad\u(\x,t))=\frac{u(r,t)}{r},
\end{equation}
and $(\mu,\nu)$ are the eigenvalues of potential gradient $D^2\phi(r,t)$ 
\begin{equation}\label{def:munu}
  \mu(r,t) = \lambda_1(D^2\phi(\x,t))=\pa_r^2\phi(r,t),\quad
  \nu(r,t) = \lambda_2(D^2\phi(\x,t))=\frac{\pa_r\phi(r,t)}{r}.
\end{equation}

The following lemma shows that the boundedness of the radial derivative $p=\pa_ru(r,t)$ is sufficient to
guarantee the boundedness of $\grad\u$. 
\begin{lemma}\label{lem:gradu}
 Consider the radial velocity field
 \eqref{eq:radial},\eqref{eq:radialzero},
 where $\displaystyle \u(\x,t)=\frac{\x}{r}u(r,t)$ and $u(0,t)=0$.   Then
  \begin{equation}\label{eq:gradup}
    \|\grad\u(\cdot,t)\|_{L^\infty}\leq\|\pa_ru(r,t)\|_{L^\infty}.
  \end{equation}
\end{lemma}
To verify \eqref{eq:gradup} recall that $\grad\u(\x,t)$ is given by the radial Hessian
\[
  \grad\u(\x,t) = D^2U(r)=q(r,t)\,\Id +
    \big(p(r,t)-q(r,t)\big)\,\frac{\x\,\x^\top}{r^2}, \qquad \left\{\begin{array}{l}
    p(r,t)=\pa_ru(r,t)\\
    \displaystyle q(r,t) =\frac{u(r,t)}{r},
    \end{array}\right.
    \]
and hence for arbitrary unit vector ${\mathbf w}$, 
\[
\langle \grad\u(\x,t) {\mathbf w},{\mathbf w}\rangle = \theta p+(1-\theta)q
\leq \max\{p,q\}, \qquad \theta=\frac{|\langle \x,{\mathbf w}\rangle|^2}{r^2}\in [0,1].
\]
Moreover, by \eqref{eq:radialzero}, $u(0,t)=0$ and hence
\[ |q(r,t)|=\frac{1}{r}\left|\int_0^r\pa_s u(s,t)\,{\d}s\right|\leq
  \|p(\cdot,t)\|_{L^\infty},\]
and \eqref{eq:gradup} follows from the last two inequalities.

\section{Spectral dynamics for radial Euler-Monge-Amp\`ere system}\label{sec:threshold}
\subsection{Thresholds for Euler-Monge-Amp\`ere system}
In this section, we aim to study the spectral dynamics \eqref{eq:pq}
of the radial EMA system. The goal is to obtain an $L^\infty$ bound on $p$.

Let us start with expressing the Monge-Amp\`ere equation \eqref{eq:MA} as
\begin{equation}\label{eq:MAspectral}
  \rho = det(\Id-D^2\phi) = \prod_{i=1}^n(1-\lambda_i(D^2\phi))=(1-\mu)
  (1-\nu)^{n-1}.
\end{equation}
From the definition \eqref{def:munu}, we observe the following
relation between $\mu$ and $\nu$
\[\mu = \pa_r(r\nu) = r\pa_r\nu + \nu.\]
Hence, we have
\[(1-\mu)(1-\nu)^{n-1}=(1-\nu)^{n}-r\pa_r\nu\,(1-\nu)^{n-1}
  =\frac{1}{nr^{n-1}}\,\pa_r\big(r^n(1-\nu)^n\big).\]
Then, the Monge-Amp\`ere equation \eqref{eq:MAspectral} amounts to
\begin{equation}\label{eq:rhonurelation}
  \pa_r\big(r^n(1-\nu)^n\big) = n r^{n-1}\rho.
\end{equation}

\begin{lemma}\label{lem:rhoprimitive}
  Let $\rho(\x,t)=\rho(r,t)$ be a radial solution of the continuity
  equation \eqref{eq:density}. Define
  \begin{equation}\label{eq:e}
    e(r,t)=\int_0^rs^{n-1}\rho(s,t)\,{\d}s.
  \end{equation}
  Then, $e$ satisfies the transport equation
  \[e'=\pa_te+u\pa_re=0.\]
\end{lemma}
\begin{proof}
  Under radial symmetry \eqref{eq:radial}, the continuity equation
  \eqref{eq:density} can be written as
  \begin{equation}\label{eq:rho}
    \pa_t\rho+\pa_r(\rho u)=-\frac{(n-1)\rho u}{r}.
  \end{equation}
  Then, we can compute
  \begin{align*}
    \pa_r(e')=&\,\pa_t\pa_re+\pa_r(u\pa_re)=
    \pa_t\big(r^{n-1}\rho\big)+\pa_r\big(r^{n-1}\rho u\big)\\
    =&\,r^{n-1}\left(\pa_t\rho+\pa_r(\rho u)+\frac{(n-1)\rho u}{r}\right)=0.
  \end{align*}
  This implies $e'(\cdot,t)$ is a constant. By definition
  $e'(0,t)=0$. Therefore,  we conclude that $e'=0$.
\end{proof}

From \eqref{eq:rhonurelation}, we get $r^n(1-\nu)^n=ne$.
Applying Lemma~\ref{lem:rhoprimitive}, we obtain
\[\big(r(1-\nu)\big)'=0.\]
Note that $r=r(t; r_0)$ is the characteristic path initiated at $r_0$, satisfying
\[r' = u(r,t),\quad r(0; r_0)=r_0.\]
This implies the dynamics of $\nu$
\begin{equation}\label{eq:nu}
  \nu' = \frac{r'}{r}(1-\nu) = \frac{u}{r}(1-\nu) = q(1-\nu).
\end{equation}

The dynamics of $(q,\nu)$ form a closed ODE system along
characteristic paths
\begin{equation}\label{eq:qnu}
  \begin{cases}
    q' = -q^2-\kappa \nu,\\
    \nu' = q(1-\nu).
  \end{cases}
\end{equation}

\begin{proposition}\label{prop:CTqnu}
  Consider the ODE system \eqref{eq:qnu} with initial condition $(q_0,
  \nu_0)$. Then, the solution remains bounded in all time if and only if
  \begin{equation}\label{eq:CTqnu}
    |q_0|<\sqrt{\kappa(1-2\nu_0)}.
  \end{equation}
  Moreover, if $\eqref{eq:CTqnu}$ is violated, there exists a finite
  time $T_c$, such that
  \[\lim_{t\to T_c^-}q(t) = -\infty,\quad
    \lim_{t\to T_c^-}\nu(t)=\begin{cases}
      \infty&\nu_0>1,\\ 1&\nu_0=1,\\ -\infty&\nu_0<1.
    \end{cases}
    \]
\end{proposition}
\begin{proof}
  Let us first consider the case when $\nu_0\geq1$. We claim that the
  solution must blows up in finite time. Suppose $(q,\nu)$ are bounded
  in any finite time. Then, we have
  \[\nu(t) = 1 + (\nu_0-1)\exp\left[\int_0^t
      q(s)\,ds\right]\geq1,\quad \forall~t\geq0.\]
  Then, we get
  \[q'\leq -q^2-\kappa,\]
  which must blow up in finite time, namely there exists a $T_c$ such
  that
  \[\lim_{t\to T_c^-}q(t) = -\infty.\]
  This leads to a contradiction.
  Furthermore, if $\nu_0>1$, we have
  \[\lim_{t\to T_c^-}\nu(t)=\infty.\]
  If $\nu_0=1$, then $\nu(t)\equiv1$. This corresponds to the case
  when $\rho(t)\equiv0$.
  
  Next, we consider the case $\nu_0<1$. Define
  \begin{equation}\label{eq:wv}
    w = \frac{q}{1-\nu},\quad v=\frac{1}{1-\nu}.
  \end{equation}
  The dynamics of $(w,v)$ forms a linear system
  \begin{align*}
    w'=&\, \frac{q'(1-\nu)+q\nu'}{(1-\nu)^2}=\frac{(-q^2-\kappa\nu)(1-\nu)+q^2(1-\nu)}{(1-\nu)^2}=\frac{-\kappa\nu}{1-\nu}=\kappa(1-v),\\
    v'=&\,\frac{\nu'}{(1-\nu)^2}=\frac{q}{1-\nu}=w.
  \end{align*}
  The trajectory is an ellipse in the $(w,v)$ phase plane. Indeed, we
  have
  \[\big(w^2+\kappa(1-v)^2\big)'=2w\cdot k(1-v)+2\kappa (1-v)\cdot
    (-w) = 0.\]
  The only possible blowup is when $v\to0$.
  Clearly, $v$ remains away from zero if and only if the initial
  condition satisfies
  \[w_0^2+\kappa(1-v_0)^2<\kappa.\]
  Expressing the condition in $(q_0,\nu_0)$, we end up with
  \eqref{eq:CTqnu}.

  If \eqref{eq:CTqnu} is violated, there exists a time $T_c$ such that
  $v(T_c)=0$. Then, we have
  \begin{align*}
    &\lim_{t\to T_c^-}\nu(t)=1-\lim_{t\to T_c^-}\frac{1}{v(t)}=-\infty,\\
    &\lim_{t\to T_c^-}q(t) = \lim_{t\to T_c^-}\frac{\nu'(t)}{1-\nu(t)} =
    -\lim_{t\to T_c^-}\big(\log (1-\nu(t))\big)'=-\infty.
  \end{align*}
\end{proof}

Next, we discuss the dynamics of $\mu$. From the Monge-Amp\`ere
equation \eqref{eq:MAspectral}, we have
\[\mu = 1-\frac{\rho}{(1-\nu)^{n-1}}.\]
Recall the dynamic of $\rho$ \eqref{eq:rho}
\[\rho' = - \rho\,\pa_ru-\frac{(n-1)\rho u}{r}=-\rho\big(p+(n-1)q\big).\]
This, together with \eqref{eq:nu}, implies
\begin{align*}
  \mu' =&\, -\frac{\rho'(1-\nu)^{n-1}+(n-1)\rho(1-\nu)^{n-2}\nu'}{(1-\nu)^{2n-2}}
  = -\frac{-\rho\big(p+(n-1)q\big)+(n-1)\rho q}{(1-\nu)^{n-1}}\\
=&\, \frac{\rho p}{(1-\nu)^{n-1}}=p(1-\mu).
\end{align*}
Therefore, the dynamics of $(p,\mu)$ also forms a closed ODE system along
characteristic paths
\begin{equation}\label{eq:pmu}
  \begin{cases}
    p' = -p^2-\kappa \mu,\\
    \mu' = p(1-\mu).
  \end{cases}
\end{equation}
Observe that it is the same as the dynamics of $(q,\nu)$ in
\eqref{eq:qnu}. We obtain the same critical threshold condition.
\begin{proposition}\label{prop:CTpmu}
  Consider the ODE system \eqref{eq:pmu} with initial condition $(p_0,
  \mu_0)$. Then, the solution remains bounded in all time if and only if
  \begin{equation}\label{eq:CTpmu}
    |p_0|<\sqrt{\kappa(1-2\mu_0)}.
  \end{equation}
    Moreover, if $\eqref{eq:CTqnu}$ is violated, there exists a finite
  time $T_c$, such that
  \begin{equation}\label{eq:pmublowup}
    \lim_{t\to T_c^-}p(t) = -\infty,\quad
    \lim_{t\to T_c^-}\mu(t)=\begin{cases}
      \infty&\mu_0>1,\\ 1&\mu_0=1,\\ -\infty&\mu_0<1.
    \end{cases}
  \end{equation}
\end{proposition}

We end up with the following sharp critical threshold result for the
radial EMA system.

\begin{theorem}\label{thm:CT}
 Let $(\rho,\u,\phi)$ be a classical solution of the EMA system
 \eqref{eqs:EMA} with radial symmetry \eqref{eq:radial}.
 \begin{itemize}
 \item If the initial condition satisfies
   \begin{equation}\label{eq:CTsub}
     |u_0'(r)|<\sqrt{\kappa(1-2\phi_0''(r))},
   \end{equation}
   for all $r>0$, then the solution $\rho$ and $\grad\u$ are uniformly
   bounded in all time.
 \item If there exists an $r>0$ such that \eqref{eq:CTsub} is
   violated, then there exists a location $r_c$ and a finite time $T_c$, such that 
   \begin{equation}\label{eq:CTblowup}
     \lim_{t\to T_c^-}u_r(r_c,t)=-\infty,\quad \lim_{t\to
       T_c^-}\rho(r_c,t)=\infty\text{ (or 0)}.
   \end{equation}
   \end{itemize}
\end{theorem}
\begin{proof}
  For subcritical initial data satisfying \eqref{eq:CTsub},
  we can apply Proposition \ref{prop:CTpmu} along all characteristic
  paths and obtain boundedness of $\|p(\cdot,t)\|_{L^\infty}$ and
  $\|\mu(\cdot,t)\|_{L^\infty}$ in all time.
  Then, uniform boundedness on $\|\grad\u(\cdot,t)\|_{L^\infty}$
  follows directly from Lemma \ref{lem:gradu}.

  To obtain boundedness of $\rho$, we recall that
  $\rho=(1-\mu)(1-\nu)^{n-1}$. Therefore, it suffies to show
  boundedness of $\nu$. Through a similar argument as in Lemma
  \ref{lem:gradu}, we have
  \[|\pa_r\phi(r,t)|=\left|\pa_r\phi(0,t)+\int_0^r\pa_r^2\phi(s,t)\,{\d}s\right|\leq
    r\|\mu(\cdot,t)\|_{L^\infty}.\]
  Hence, $\|\nu(\cdot,t)\|_{L^\infty}\leq\|\mu(\cdot,t)\|_{L^\infty}$.
  Consequently,
  $\|\rho(\cdot,t)\|_{L^\infty}\leq\|\mu(\cdot,t)\|_{L^\infty}^n$ is
  bounded.

  For supercritical initial data, suppose \eqref{eq:CTsub} is violated
  at $r=r_0>0$. Then, applying Proposition \ref{prop:CTpmu}, the
  solution of the ODE system \eqref{eq:pmu} with initial condition
  $p(0)=u'_0(r_0)$ and $\mu(0)=\phi''_0(r_0)$ becomes unbounded in
  a finite time $T_c$, at the location $r_c=r(T_c; r_0)$.
  Moreover, if solution is smooth in $[0,T_c)$, \eqref{eq:pmublowup}
  directly implies \eqref{eq:CTblowup}. In particular, the case 
  $\rho(r_c,T_c)=0$ only happens if $\mu_0=1$, or equivalently
  $\rho_0(r_0)=0$.
\end{proof}

\begin{remark}\label{rem:rhom}
  According to \eqref{eq:CTqnu} and \eqref{eq:CTpmu}, global solutions of their respective ODEs   require that $\nu_0>\nicefrac{1}{2}$ and, respectively, $\mu>\nicefrac{1}{2}$, hence a  global smooth solution of    \eqref{eq:MAspectral} requires
  \[\rho_0=det(\Id-D^2\phi)= (1-\mu_0)(1-\nu_0)^{n-1}>
    \frac{p_0^2+\kappa}{2\kappa}\left(\frac{q_0^2+\kappa}{2\kappa}\right)^{n-1}
  \ge \frac{1}{2^n}.\]
  This recovers the necessary  lower-bound, $\rho_0>\nicefrac{1}{2}$, for global regularity in the case $n=1$. Thus, $\rho_0(r_0) <2^{-n}$ will necessarily leads to formation of shock discontinuities and in particular, a vacuous state of $\rho_0$  leads to formation of (non-physical) shocks.   
  On the other hand, if
 $\rho_0$ is not far away from the equilibrium state $\rho_0=1$, such
 that $\mu_0>\nicefrac{1}{2}$ and $\nu_0>\nicefrac{1}{2}$,
 we can always find $p_0$ and $q_0$ small
 enough, such that \eqref{eq:CTqnu} and \eqref{eq:CTpmu} hold.
\end{remark}

\subsection{A comparison with  Euler-Poisson equations}\label{sec:compEP}

In this section, we compare our critical threshold result for the
EMA system \eqref{eqs:EMA} with the Euler-Poisson equations
\eqref{eqs:EP}, under radial symmetry.

A sharp critical threshold was obtained in \cite{tan2021eulerian}
for the radial Euler-Poisson equations, following a similar procedure.
We summarize the result here for the sake of self-consistency,
using the same notations $(p,q,\mu,\nu)$ as defined in \eqref{def:pq}
and \eqref{def:munu}.

The Poisson equation \eqref{eq:Poisson} can be expressed as
\begin{equation}\label{eq:EPrho}
  -(\mu+(n-1)\nu)=\widetilde\rho-1,
\end{equation}
which implies
\[\pa_r\big(r^n(1-n\nu)\big) = n r^{n-1}\widetilde\rho.\]
Applying Lemma \ref{lem:rhoprimitive} with  $e=r^n(\frac1n-\nu)$, we obtain
\[\big(r^n(1-n\nu)\big)'=ne'=0.\]
It yields
\[ \nu' = q(1-n\nu). \]
Hence, the dynamics of $(q,\nu)$ reads
\begin{equation}\label{eq:EPqnu}
  \begin{cases}
    q' = -q^2-\kappa \nu,\\
    \nu' = q(1-n\nu).
  \end{cases}
\end{equation}
In contrast to \eqref{eq:qnu}, the dynamics depends on the dimension
$n$. The global behaviors are surprisingly different.
\begin{proposition}[{\cite[Theorem 3.15]{tan2021eulerian}}]\label{prop:EPqnu}
  Let $n\geq2$. consider the ODE system \eqref{eq:EPqnu} with bounded
  initial data $(q_0, \nu_0<\frac1n)$. Then, the solution $(q,\nu)$ remains
  bounded in all time.
\end{proposition}
Note that from the definition \eqref{eq:e}, $e_0(r)\geq0$, where
the inequality holds in the trivial case where $\widetilde\rho_0(s)=0$ for
$s\in[0,r]$. Therefore, $\nu_0<\frac1n$ holds for generic initial
data. The indicates different behaviors as the EMA system, where
blowup can happen in the $(q,\nu)$ dynamics, as long as
\eqref{eq:CTqnu} is violated.

The dynamics of $(p,\mu)$ however, is less understood for the
Euler-Poisson equations. Indeed, we can calculate from
\eqref{eq:EPrho} and \eqref{eq:EPqnu}
\begin{align*}
  \mu' =&\, -\widetilde\rho\,'-(n-1)\nu'=\widetilde\rho\left(p+(n-1)q\right)-(n-1)q(1-n\nu)\\
  =&\, p(1-\mu) + (n-1)\big[-p\nu-q(\mu-\nu)\big].
\end{align*}
This does not yield a closed ODE system on $(p,\mu)$, except when
$n=1$, where the Poisson equation \eqref{eq:Poisson} coincides with the
Monge-Amp\`ere equation \eqref{eq:MA}. Whether there is an explicit
threshold condition that leads to a global bound for the
$(p,q,\mu,\nu)$ dynamics for the radial Euler-Poisson equations is open.

The explicit result in Theorem \ref{thm:CT} indicates that the EMA
system has some special structures compared with the Euler-Poisson
equations, despite of being fully nonlinear. It is related to the
geometric structure of the Monge-Amp\`ere equation, which will be
discussed in Section \ref{sec:geometric}.

\section{Global wellposedness}\label{sec:GWP}
The local and global wellposedness theory for the EMA system
\eqref{eqs:EMA} in Sobolev space $H^s$ has been
established by Loeper in \cite{loeper2005quasi}, using energy
estimates. The theory requires a smallness assumption on the potential
$\phi$ to handle the nonlinearity from the Monge-Amp\`ere
equation.

We now establish a global wellposedness theory for the EMA system with
radial symmetry. We make use of the critical threshold condition, and
do not require any smallness assumptions.

Let $\U$ be a vector-valued radial function defined as
\begin{equation}\label{eq:U}
    \U(\x,t)=\begin{bmatrix}p(|\x|,t)\\ q(|\x|,t)\\ \mu(|\x|,t)\\
      \nu(|\x|,t)\end{bmatrix}
  =\begin{bmatrix}\pa_ru(|x|,t)\\ \frac{u(|\x|,t)}{|x|}\\ \pa_r^2\phi(|\x|,t)\\
    \frac{\pa_r\phi(|\x|,t)}{|x|}\end{bmatrix}.
\end{equation}
From the dynamics \eqref{eq:qnu} and \eqref{eq:pmu}, we know $\U$
  satisfies
  \[
    \pa_t\U+(\u\cdot\grad)\U=\F(\U),\quad
    \F(\U)=\begin{bmatrix}
      -U_1^2-\kappa U_3\\ -U_2^2-\kappa U_4\\
      U_1(1-U_3)\\ U_2(1-U_4).
    \end{bmatrix}.
  \]
  Equivalently, we can write
  \begin{equation}\label{eq:bigU}
    \pa_tU_i+\div(U_i\u)=\widetilde{F}_i(\U),\quad
    \forall~i=1,2,3,4,
  \end{equation}
  with a nonlinear force $\widetilde{\F}$ which depends quadratically on
  $\U$
  \begin{equation}\label{eq:Ftilde}
    \widetilde{\F}(\U):=\F(\U)+(\div\u)\,\U=\begin{bmatrix}
      -U_1^2-\kappa U_3+U_1\big(U_1+(n-1)U_2\big)\\
      -U_2^2-\kappa U_4+U_2\big(U_1+(n-1)U_2\big)\\
      U_1(1-U_3)+U_3\big(U_1+(n-1)U_2\big)\\
      U_2(1-U_4)+U_4\big(U_1+(n-1)U_2\big)
    \end{bmatrix}.
  \end{equation}
  Here, we have used
  \begin{equation}\label{eq:divu}
    \div\u=\sum_{i=1}^n\lambda_i(\grad\u)=U_1+(n-1)U_2.
  \end{equation}

Let us first state a local wellposedness theory, as well as regularity criteria.
\begin{theorem}\label{thm:local}
  Consider the EMA system \eqref{eqs:EMA} with radial
  symmetry \eqref{eq:radial}. $\U$ is defined in \eqref{eq:U}.
  Suppose the initial condition $\U_0\in H^s(\R^n)$, for $s>\frac{n}{2}$.
  Then, there exists a time $T>0$ such that the solution
  \begin{equation}\label{eq:regularity}
    \U\in C([0,T],  H^{s}(\R^n))^4.
  \end{equation}
  Moreover, the life span $T$ can be extended as long as
  \begin{equation}\label{eq:BKM}
    \int_0^\top\|\U(\cdot,t)\|_{L^\infty}\,{\d}t<+\infty.
  \end{equation}
\end{theorem}

\noindent
The local existence result follows from the standard energy method in which one obtain a closure of $H^s$ estimates for $s>n/2$ so that $H^s(\R^n) \subset L^\infty(\R^n)$, as long as the Beale-Kato-Majda like condition 
$\displaystyle \int_0^T \|\nabla\cdot\u(\cdot,t)\|_{L^\infty}{\d}t <\infty$ holds, e.g., \cite{lin2000hydrodynamic}. For completeness, we outline the details below. 
\begin{proof}

  Given any $s\geq0$, denote $\Lambda^s = (-\Delta)^{s/2}$ as the
  fractional Laplacian operator. Define energy $Y_s(t)$ as
  \[Y_s(t)=\frac{1}{2}\|\U\|_{H^s(\R^n)}^2=\frac{1}{2}\|\U\|_{L^2(\R^n)}^2+\frac{1}{2}\|\Lambda^s\U\|_{L^2(\R^n)}^2.\]
  
  The $L^2$ energy can be estimated by
  \begin{align*}
    \frac12\frac{d}{dt}\|\U\|_{L^2}^2=&\,
    -\int_{\R^n}U_i\cdot\pa_{x_j}(U_iu_j)\,{\d}x+\int_{\R^n}U_i\cdot \widetilde{F}_i(\U)\,{\d}x\\
    \leq&\,\int_{\R^n}\pa_{x_j}\left(\frac{U_i^2}{2}\right)\cdot u_j\,{\d}x+\|U_i\|_{L^2}\cdot\|\widetilde{F}_i(\U)\|_{L^2}\\
    \leq&\,-\int_{\R^n}\pa_{x_j}u_j\cdot\frac12U_i^2\,{\d}x +C\|\U\|_{L^2}\cdot(1+\|\U\|_{L^\infty})\|\U\|_{L^2}.\\
    \lesssim&\,
    (1+\|\div\u\|_{L^\infty}+\|\U\|_{L^\infty})\|\U\|_{L^2}^2
    \lesssim  (1+\|\U\|_{L^\infty})\|\U\|_{L^2}^2.
  \end{align*}
  Here, we use Einstein summation convention and drop the summation on $i$
  and $j$ for simplicity. We also use the notation $\lesssim$, where
  $A\lesssim B$ means $A\leq CB$, with a constant $C$ which might
  depend on parameters (like $n$, $s$, etc.).
  In the penultimate line, we make use of the quadratic
  dependence of $\widetilde{\F}$ on $\U$ in \eqref{eq:Ftilde}. Apply H\"older inequality and
  get
  \[\|\widetilde{\F}(\U)\|_{L^2}\lesssim (1+\|\U\|_{L^\infty})\|\U\|_{L^2}.\]
  The last inequality is due to \eqref{eq:divu}.
  
  Next, we estimate the $\dot{H}^s$ energy.
  Apply $\Lambda^s$ to \eqref{eq:bigU}, multiply by $\Lambda^s\U$, and
  integrate in $\R^n$. We obtain
  \[\frac12\frac{d}{dt}\|\Lambda^s\U\|_{L^2}^2=
    -\int_{\R^n}\Lambda^sU_i\cdot\Lambda^s\pa_{x_j}(U_iu_j) \,{\d}x
    +\int_{\R^n}\Lambda^sU_i\cdot \Lambda^s(\widetilde{F}_i(\U))\,{\d}x=I+II.\] 
  To estimate $I$, we use a Kato-Ponce type commutator estimate
  \cite{kato1988commutator} and get
  \begin{align*}
    I=&\, -\int \Lambda^sU_i\cdot u_j\,\Lambda^s\pa_{x_j}U_i\,{\d}x
        -\int \Lambda^sU_i\cdot \big[\Lambda^s\pa_{x_j}, u_j\big]U_i\,{\d}x\\
    \leq&\, \int\pa_{x_j}u_j\cdot\frac12\,(\Lambda^sU_i)^2 \,{\d}x
       +\|\Lambda^sU_i\|_{L^2}\left\|\big[\Lambda^s\pa_{x_j}, u_j\big]U_i\right\|_{L^2}\\
    \leq&\, \frac{1}{2}\|\div\u\|_{L^\infty}\|\Lambda^s\U\|_{L^2}^2
        +C \|\Lambda^sU_i\|_{L^2}\big(\|\grad u_j\|_{L^\infty}\|\Lambda^sU_i\|_{L^2}
        +\|\Lambda^{s+1}u_j\|_{L^2}\|U_i\|_{L^\infty}\big)\\
    \lesssim&\, \|\grad\u\|_{L^\infty}\|\Lambda^s\U\|_{L^2}^2+\|\Lambda^s\U\|_{L^2}\|\U\|_{L^\infty}\|\Lambda^s(\grad\u)\|_{L^2}.
  \end{align*}
  Furthermore, using \eqref{eq:divu}, we have
  \[ \|\Lambda^s(\grad\u)\|_{L^2}\lesssim \|\Lambda^s(\div\u)\|_{L^2}
    \leq\|\Lambda^sU_1\|_{L^2}+(n-1)\|\Lambda^sU_2\|_{L^2}\lesssim
    \|\Lambda^s\U\|_{L^2}.\]
  Applying the estimate above and \eqref{eq:gradup}, we obtain
  \begin{equation}\label{eq:energyI}
    I\lesssim \|\U\|_{L^\infty}\|\Lambda^s\U\|_{L^2}^2.
  \end{equation}
  The $II$ term can be estimated as follows
  \begin{equation}\label{eq:energyII}
    II\leq
    \|\Lambda^s\U\|_{L^2}\|\Lambda^s(\widetilde{\F}(\U))\|_{L^2}\lesssim
    (1+\|\U\|_{L^\infty})\|\Lambda^s\U\|_{L^2}^2.
  \end{equation}
  Here, we have used the quadratic dependence of $\widetilde{\F}$ on $\U$
  in \eqref{eq:Ftilde} again, and apply fractional Leibniz rule to get
  \[\|\Lambda^s(\widetilde{\F}(\U))\|_{L^2}\lesssim
    (1+\|\U\|_{L^\infty})\|\Lambda^s\U\|_{L^2}.\]
  Combining \eqref{eq:energyI} and \eqref{eq:energyII}, we end up with
  \[\frac12\frac{d}{dt}\|\Lambda^s\U\|_{L^2}^2\lesssim
    (1+\|\U\|_{L^\infty})\|\Lambda^s\U\|_{L^2}.\]

  From the $L^2$ and $\dot{H}^s$ energy estimates, we get
  \[\frac{d}{dt}Y_s(t)\lesssim (1+\|\U(\cdot,t)\|_{L^\infty})\,Y_s(t).\]
  Local wellposedness follows from standard Sobolev embedding, for any
  $s>\frac{n}{2}$. Moreover, we apply Gr\"onwall inequality
  \[Y_s(t)\leq Y_s(0)\exp\left[C\int_0^t
      (1+\|\U(\cdot,s)\|_{L^\infty})\,ds\right].\]
  Hence, $Y_s(t)$ is bounded as long as \eqref{eq:BKM} holds.
\end{proof}

Theorem \ref{thm:CT} provides sufficient and necessary conditions to
ensure the regularity criterion \eqref{eq:BKM}. Hence,
our main Theorem \ref{thm:main} is a direct consequence of Theorems
\ref{thm:CT} and \ref{thm:local}.

\section{A geometric approach}\label{sec:geometric}
In this section, we provide an alternative way to study the global
wellposeness of the EMA system \eqref{eqs:EMA},
taking advantage of the geometric structure of the system.

Let us start with the definition and notation for the pushforward
mapping.
\begin{definition}[Pushforward]
  Let $\Omega\subset\R^n$. A measurable mapping $T: \Omega\to\R^n$ is called a
  pushforward from a measure $\mu$ in $\Omega$ to a measure $\nu$ in
  $T(\Omega)$, if for any measurable test function $f$,
  \[\int_{\Omega}f \circ T~d\mu = \int_{T(\Omega)} f~d\nu.\]
  We use the notation $T_\sharp\,\mu=\nu$. Moreover, if $d\mu = \rho_1(\x)d\x$ and
  $d\nu = \rho_2(\x)d\x$, we denote $T_\sharp\, \rho_1 = \rho_2$.
\end{definition}

The key ingredient is to link the solution of the Monge-Amp\`ere
equation to a pushforward mapping.
\begin{lemma}[Monge-Amp\`ere
equation represented as pushforward]\label{lem:PFMA}
  Let $\psi_t$ is a solution of the Monge-Amp\`ere equation
  \[det\big(D^2\psi_t(\x)\big)=\rho(\x,t).\]
  Then, $\grad\psi_t$ is a pushforward from $\rho(\x,t)d\x$ to the
  Lebesgue measure $d\x$, namely
  \begin{equation}\label{eq:MApf}
    (\grad\psi_t)_\sharp\,\rho(\cdot,t) = 1.
  \end{equation}
\end{lemma}
The proof of the lemma can be done by a simple change of variable
formula. 
We notice that the representation \eqref{eq:MApf} makes sense as
long as $\rho(\cdot,t)$ is a measure.
If we further assume that the density $\rho(\cdot,t)$ is bounded and away from
vacuum
\begin{equation}\label{eq:rhononvac}
  0<\rho_{\min}(t)\leq \rho(\cdot,t)\leq\rho_{\max}(t)<+\infty,
\end{equation}
then $\grad\psi_t$ is a diffeomorphism.

Once we find $\psi_t$ that solves \eqref{eq:MApf}, the solution of the 
Monge-Amp\`ere equation \eqref{eq:MA} can be expressed as
\begin{equation}\label{eq:phipsi}
  \phi(\x,t) = \frac{|\x|^2}{2}-\psi_t(\x).
\end{equation} 

In order the construct the pushforward mapping $\grad\psi_t$ that
satisfies \eqref{eq:MApf}, we make use of the characteristic flow
$\X_t(\x)$, defined as
\begin{equation}\label{eq:characterstic}
  \pa_t\X_t(\x)=\u(\X_t(\x),t),\quad \X_0(\x)=\x,
\end{equation}
where $\u$ is the velocity field. $\X_t$ can be viewed as a
pushforward mapping.

\begin{lemma}[Characteristic flow represented as pushforward]\label{lem:PF}
  Suppose $\rho$ satisfies \eqref{eq:density}, with a Lipschitz flow
  $\u$. Then, $\X_t$ is a diffeomorphism. It satisfies
  \begin{equation}\label{eq:rhoPF}
    (\X_t)_\sharp\,\rho_0=\rho(\cdot,t).
  \end{equation}
\end{lemma}
The proof of the lemma is elementary if the flow $\u$ is Lipschitz.
Moreover, $\X_t$ is a diffeomorphism. We denote its inverse mapping
$\X_t^{-1}$.

Let us define another mapping $\bG$, which pushforwards the
Lebesgue measure $d\x$ to $\rho_0 d\x$, namely
\begin{equation}\label{eq:Gamma}
  \bG_\sharp \,1=\rho_0.
\end{equation}
We further define
\[\tilde{\X}_t:=\X_t\circ\bG.\]
From \eqref{eq:rhoPF} and \eqref{eq:Gamma}, we get
$(\tilde{\X}_t)_\sharp\, 1=\rho(\cdot,t)$.
Then, if $\bG$ is invertible, we have
\begin{equation}\label{eq:Xtildeinv}
  \left(\tilde{\X}_t^{-1}\right)_\sharp\,\rho(\cdot,t)=
  \big(\bG^{-1}\circ \X_t^{-1}\big)_\sharp\,\rho(\cdot,t)=1.
\end{equation}
Hence, if $\widetilde\X_t^{-1}$ has a gradient form, the corresponding
stream function is a solution of the Monge-Amp\`ere equation \eqref{eq:MApf}.

The following lemma show that $\widetilde\X_t^{-1}$ indeed has a gradient form,
under the radial symmetry.
\begin{lemma}\label{lem:Xinvgrad}
  Let $(\rho, \u, \phi)$ be a solution of \eqref{eqs:EMA} with radial
  symmetry \eqref{eq:radial}.
  Assume $\u$ is Lipschitz, and the initial density $\rho_0$ satisfies
  \begin{equation}\label{eq:rho0nonvac}
    0<\rho_{\min}(0)\leq\rho_0(\cdot)\leq\rho_{\max}(0)<+\infty.
  \end{equation}
  Then, there exists a pushforward mapping $\bG$ satisfying
  \eqref{eq:Gamma}.
  Also, there exists a radial function $\psi_t$, defined in \eqref{eq:psi}, such that
\begin{equation}\label{eq:Xtinvgrad}
  \widetilde\X_t^{-1}(\x)=\grad\psi_t(\x).
\end{equation}
 Moreover, $\bG$ and $\grad\psi_t$ are diffeomorphism.
\end{lemma}
\begin{proof}
  First, we construct $\bG$. From Lamma \ref{lem:PFMA}, we can write \eqref{eq:Gamma}
  equivalently as
  \[det\big(\grad\bG^{-1}(\x)\big)=\rho_0.\]
  Under radial symmetry \eqref{eq:radial}, the mapping $\bG$ takes
  the following form
  \[\bG(\x)=\frac{\x}{r}\,\Gamma(r),\quad
  \bG^{-1}(\x)=\frac{\x}{r}\,\Gamma^{-1}(r), \quad r=|\x|.\]
  Indeed, apply Lemma \ref{lem:eigenvalue} with $f=\Gamma^{-1}$ and get
  \[
    det\big(\grad\bG^{-1}(\x)\big)=
    (\Gamma^{-1})'(r)\cdot\left(\frac{\Gamma^{-1}(r)}{r}\right)^{n-1}=\frac{\frac{d}{dr}\big(\Gamma^{-1}(r)^n\big)}{nr^{n-1}}=\rho_0.
  \]
  The last equality holds if we define 
  \[\Gamma^{-1}(r)=\left[\int_0^rns^{n-1}\rho_0(s)\,ds\right]^{\frac1n}.\]
  This completes the construction of $\bG^{-1}$. Moreover, as $\rho_0$
  satisfies \eqref{eq:rho0nonvac}, $\bG^{-1}$ is a diffeomorphism. So
  does $\bG$.

  Next, we construct $\psi_t$.
  The dynamics of the mapping $\widetilde\X_t$ reads
  \begin{equation}\label{eq:Xtildedyn}
   \pa_t\widetilde\X_t(\x)=\u(\widetilde\X_t(\x),t),\quad \widetilde\X_0(\x)=\bG(\x),
  \end{equation}
  Under radial symmetry \eqref{eq:radial}, $\widetilde\X_t$ take the form
  \[\widetilde\X_t(\x)=\frac{\x}{r}\,R_t(r), \quad
  \widetilde\X_t^{-1}(\x)=\frac{\x}{r}\,R_t^{-1}(r), \quad r=|\x|,\]
where $R_t$ satisfies
\[\pa_tR_t(r) = u(R_t(r),t),\quad R_0(r)=\Gamma^{-1}(r).\]

We define a radial function $\psi_t$ as follows
\begin{equation}\label{eq:psi}
  \psi_t(\x)=\psi_t(r)=\int_0^r R_t^{-1}(s)\,ds.
\end{equation}
Then, we can verify that $\psi_t$ satisfies \eqref{eq:Xtinvgrad}
\[\grad\psi_t(\x)=\pa_r\psi_t(r)\, \frac{\x}{r}=
  \frac{\x}{r}\,R_t^{-1}(r)=\widetilde\X_t^{-1}(\x).\]
Moreover, as $\u$ is a Lipschitz flow, we have
\begin{align*}
  &\rho_{\max}(t)\leq
  \rho_{\max}(0)\exp\left(\int_0^t\|\div\u(\cdot,s)\|_{L^\infty}\,ds\right)<+\infty,\\
  &\rho_{\min}(t)\geq
    \rho_{\min}(0)\exp\left(-\int_0^t\|\div\u(\cdot,s)\|_{L^\infty}\,ds\right)>0.
\end{align*}
This verifies the condition \eqref{eq:rhononvac}. Hence, $\grad\psi_t$
is a diffeomorphism.
\end{proof}

Combining \eqref{eq:Xtinvgrad} with \eqref{eq:Xtildeinv}, we find a
solution of \eqref{eq:MApf}, defined in \eqref{eq:psi}. This allows us
to obtain an explicit expression of the characteristic path $\X_t$.

\begin{proposition}
  Under the same assumptions as Lemma \ref{lem:Xinvgrad},
  the characteristic flow $\X_t$ satisfies
  \begin{equation}\label{eq:Xt}
    \X_t(\x) = \big(\x-\grad\phi_0(\x)\big)+\grad\phi_0(\x)\cos (\sqrt\kappa
    t)+\u_0(\x)\frac{\sin(\sqrt\kappa t)}{\sqrt\kappa}.
  \end{equation}
\end{proposition}
\begin{proof}
 Let us first calculate
  \begin{align}
    \pa_t^2\X_t(\x)=&\,\pa_t\big(\u(\X_t(\x),t)\big)
    =\pa_t\u(\X_t(\x),t)+\grad\u(\X_t(\x),t) \pa_t\X_t(\x)\nonumber\\
    =&\, \pa_t\u(\X_t(\x),t)+\u(\X_t(\x),t)\cdot\grad\u(\X_t(\x),t)
     =-\kappa\grad\phi(\X_t(\x),t).\label{eq:Xtt}
  \end{align}
  Then apply the relation \eqref{eq:phipsi} and get
  \begin{equation}\label{eq:gradphi}
    \grad\phi(\X_t(\x),t)=\X_t(\x)-\grad\psi_t(\X_t(\x))=\X_t(\x)-\bG^{-1}(\x).
  \end{equation}
  Here, we have used \eqref{eq:Xtinvgrad}, so that
  \[\grad\psi_t\circ\X_t=\widetilde\X_t^{-1}\circ\X_t=\bG^{-1}\circ\X_t^{-1}\circ\X_t=\bG^{-1}.\]
  Therefore, $\X_t$ satisfies the following second order equation
  \[
      \pa_t^2\X_t(\x)=-\kappa\X_t(\x)+\kappa\bG^{-1}(\x),\quad
      \X_0(\x)=\x,\quad \pa_t\X_0(\x)=\u_0(\x).
    \]
    It can be solved explicitly, resulting
    \[   \X_t(\x) = \bG^{-1}(\x)+(\x-\bG^{-1}(\x))\cos (\sqrt\kappa
      t)+\u_0(\x)\frac{\sin(\sqrt\kappa t)}{\sqrt\kappa}.\]
    Moreover,  we apply \eqref{eq:gradphi} with $t=0$ and obtain
    $\bG^{-1}(\x)=\x-\grad\phi_0(\x)$. It leads to the formula \eqref{eq:Xt}.
\end{proof}

\begin{corollary}[Energy conservation]\label{cor:localenergy}
  Given any bounded set $\Omega\subset\R^n$, define energy
  \[E(t) =  \frac{1}{2}\int_{\X_t(\Omega)}\rho(\x,t)
    \big(|\u(\x,t)|^2+\kappa|\grad\phi(\x,t)|^2\big)\,d\x.\]
  Then, $E(t)$ is conserved in time.
\end{corollary}
\begin{proof}
  First, we apply Lemma \ref{lem:PF} and write
  \begin{equation}\label{eq:localenergy}
    E(t)=\frac{1}{2}\int_\Omega\big(|\pa_t\X_t(\x)|^2
    +\kappa|\grad\phi(\X_t(\x),t)|^2\big)\rho_0(\x)\,d\x.
  \end{equation}
  Then,
  \begin{align*}
    E'(t) = &\, \int_\Omega\Big[\pa_t\X_t(\x)\cdot\pa_t^2\X_t(\x)+
    \kappa\grad\phi(\X_t(\x),t)\cdot \pa_t\grad\phi(\X_t(\x),t)\Big]\rho_0(\x)\,d\x\\
= &\, \int_\Omega \pa_t\X_t(\x)\cdot\Big[\pa_t^2\X_t(\x)+
    \kappa\grad\phi(\X_t(\x),t)\Big]\rho_0(\x)\,d\x=0.
  \end{align*}
  Here, in the penultimate equality, we apply \eqref{eq:gradphi} and
  get $\pa_t\grad\phi(\X_t(\x),t)=\pa_t\X_t(\x)$. The last equality
  follows from \eqref{eq:Xtt}.
\end{proof}

Taking spatial gradient of \eqref{eq:Xt} would yield
  \begin{equation}\label{eq:Xtgrad}
    \grad\X_t(\x) =\big(\Id-D^2\phi_0(\x)\big)+D^2\phi_0(\x)\cos (\sqrt\kappa
    t)+\grad\u_0(\x)\frac{\sin(\sqrt\kappa t)}{\sqrt\kappa}.
  \end{equation}
We can recover the critical threshold condition that we obtained
through the analysis of the spectral dynamics.

\begin{theorem}
  $\grad\X_t(\x)$ remains positive definite in all time, if and only if
  the initial condition satisfies \eqref{eq:CTsub}.
\end{theorem}
\begin{proof}
  Since $\grad\X_t(\x)$ is symmetric, it is positive definite if and
  only if the eigenvalues $\lambda_i(\grad\X_t(\x))>0$.

  From Lemma \ref{lem:eigenvalue}, $\grad\X_t(\x),
  D^2\phi_0(\x)$ and $\grad\u_0(\x)$ all share the same
  eigenvectors. Therefore, \eqref{eq:Xtgrad} implies
  \[\lambda_i(\grad\X_t(\x)) =\big(1-\lambda_i(D^2\phi_0(\x))\big)+\lambda_i(D^2\phi_0(\x))\cos (\sqrt\kappa
    t)+\lambda_i(\grad\u_0(\x))\frac{\sin(\sqrt\kappa
      t)}{\sqrt\kappa}.\]
  Hence, $\lambda_i(\grad\X_t(\x))>0$ if and only if
  \[\lambda_i(D^2\phi_0(\x))^2+\frac{\lambda_i(\grad\u_0(\x))^2}{\kappa}<
    \big(1-\lambda_i(D^2\phi_0(\x))\big)^2,\]
  or equivalently
  \[\lambda_i(\grad\u_0(\x))^2<\kappa\left(1-2\lambda_i(D^2\phi_0(\x))\right).\]
  This is precisely \eqref{eq:CTpmu} and \eqref{eq:CTqnu} for $i=1,2$ respectively.

  Finally, the equivalency to \eqref{eq:CTsub} follows through the
  same argument in Theorem \ref{thm:CT}.
\end{proof}

\section{Beyond radial symmetry: the 2D system with swirl}\label{sec:nonradial}
We have established the global wellposedness theory for the EMA system
with radial symmetry \eqref{eq:radial}. A natural question is what
happens we do not impose radial symmetry.
In this section, we briefly discuss potential extensions of our theory
to more general data.

A major difficulty of implementing the spectral dynamics analysis to
the general data is, $\grad\u$ and $\grad\F$ do not necessarily share
the same eigenvectors, so the forcing term in \eqref{eq:lambF} can be
hard to control.

\subsection{2D radial EMA system with swirl}
Let us consider the following type of solutions in 2D
\begin{equation}\label{eq:radialswirl}
  \rho(\x,t)=\rho(r,t),\quad
  \u(\x,t)=\u_1(\x,t)+\u_2(\x,t)=\frac{\x}{r}u(r,t)+\frac{\x^\perp}{r}\Theta(r,t),
  \end{equation}
where $\Theta$ characterizes the rotation. Under the setup, although $\grad\u$
does not share eigenvectors as $\grad\F$, the component $\grad\u_1$
does. In fact, we can decompose $\grad\u$ by the symmetric part
$\grad\u_1$ and anti-symmetric part $\grad\u_2$ and study their spectral
dynamics separately. Elementary calculation yields the dynamics of
$(p,q,\mu,\nu)$ together with $(\Theta_r, \frac{\Theta}{r})$ as follows
\begin{equation}\label{eq:six}
    \begin{cases}
    q' = -q^2-\kappa \nu+\left(\frac{\Theta}{r}\right)^2,\\
    \nu' = q(1-\nu),\\
    \left(\frac{\Theta}{r}\right)'=-2q\frac{\Theta}{r},
  \end{cases}
  \text{and}\quad
    \begin{cases}p' = -p^2-\kappa \mu+2\Theta_r\frac{\Theta}{r}-\left(\frac{\Theta}{r}\right)^2,\\
    \mu' = p(1-\mu),\\
    \Theta_r'=-(p+q)\Theta_r-(p-q)\frac{\Theta}{r}.
  \end{cases}
\end{equation}
Global wellposedness follows from the solvability of the closed ODE
system, with 6 variables.
\begin{definition}\label{def:Sigma}
  Define a set $\Sigma\subset\R^6$ so that
  $\sigma_0=\Big(p_0,q_0,(\Theta_r)_0, (\frac{\Theta}{r})_0, \mu_0,\nu_0\Big)\in\Sigma$
  if and only if the ODE system \eqref{eq:six} with initial condition
  $\sigma_0$ is bounded globally in time.
\end{definition}

Clearly, if the initial data is subcritical, satisfying
\eqref{eq:2DCTcond}, the boundedness of $\sigma(t)$ implies the
boundedness of $\grad\u$. Then the solution is globally regular.
This finishes the proof of Theorem \ref{thm:2D}.

A natural question is, whether we can find an explicit formulation of
the subcritical region $\Sigma$, similar as what we have done for the system
without swirl. So, we shall examine the ODE system \eqref{eq:six}.

Note that the dynamics of $(q,\nu, \frac{\Theta}{r})$ form a closed
system. Compared with \eqref{eq:qnu}, we observe that the presence of
$\frac{\Theta}{r}$ helps to avoid $q\to-\infty$.
This phenomenon is known as \emph{rotation prevents finite-time
  breakdown}, which has been studied in \cite{liu2004rotation}.
More precisely, we state the following result.

\begin{proposition}
  Consider the dynamics $(q,\nu, \frac{\Theta}{r})$ with initial
  conditions $\nu(0)<1$ and $\tfrac{\Theta}{r}(0)\neq0$. Then,
  the solution $(q,\nu, \frac{\Theta}{r})$ remains bounded in all time.
\end{proposition}
\begin{proof}
  We follow Proposition \ref{prop:CTqnu} and define $(w,v)$ as
  \eqref{eq:wv}.
  The dynamics reads
  \[
    \begin{cases}
    w' = \kappa(1-v) + \left(\frac{\Theta}{r}\right)^2v,\\
    v' = w.
  \end{cases}
  \]
  To understand the influence from $\frac{\Theta}{r}$, we observe the
  following conserved quantity
  \[\left(\frac{\Theta}{r} v^2\right)'=
    -2q \frac{\Theta}{r}\cdot v^2 + \frac{\Theta}{r}\cdot 2v w
    =-\frac{\Theta}{r}\cdot 2v(qv-w)=0.\]
  This implies
  \[\frac{\Theta}{r}  = C_0 v^{-2},\quad \text{where}\quad
    C_0=\frac{\Theta}{r}(0) v(0)^2\neq0.\]
  It leads to the following closed system for $(w,v)$
  \[
    \begin{cases}
    w' = \kappa(1-v) + C_0^2v^{-3},\\
    v' = w.
  \end{cases}
  \]
  We obtain an invariant quantity
  \[\Big(w^2+\kappa(1-v)^2+C_0^2v^{-2}\Big)'=2w\cdot
    \big(\kappa(1-v)+C_0^2v^{-3}\big)+
    \big(-2\kappa (1-v)-2C_0^2v^{-3}\big)\cdot
    w = 0.\]
  So, we have
  \[w^2+\kappa(1-v)^2+C_0^2v^{-2}=C:=w_0^2+\kappa(1-v_0)^2+C_0^2v_0^{-2},\]
  where the constant $C>0$ is a finite number when $\nu(0)<1$.
  Clearly, $w$ is bounded with $|w|\leq\sqrt{C}$. Also,
  $\kappa(1-v)^2+C_0^2v^{-2}\leq C$ implies $v>0$ and $v$ is bounded.
  The boundedness of $(q,\nu,\frac{\Theta}{r})$ then follows as a
  direct consequence.
\end{proof}

Under radial symmetry \eqref{eq:radial}, the dynamics of $(p,\mu)$ is
the same as $(q,\nu)$. This is however not the case with swirl. In
fact, the dynamics of $(p,\nu,\Theta_r)$ does not even form a closed
system. It is unclear whether there is an explicit critical threshold condition on the
initial data that leads to the boundedness of the six quantities.

\subsection{General non-symmetric data}
Consider the EMA system \eqref{eqs:EMA} with general initial data. We found it difficult to
trace its spectral dynamics since the eigen-structure of  the gradient force is not accessible through any obvious time-invariant quantities. We shall
briefly discuss the alternative geometric approach \cite{brenier2004geometric, loeper2005quasi}.

Without radial symmetry, the flow map $\widetilde\X_t^{-1}$ might not have
a gradient form. Lemma \ref{lem:Xinvgrad} no longer holds. To find a
solution $\psi_t$ for \eqref{eq:MApf}, we make use of the celebrated
polar factorization by Bernier \cite{brenier1991polar}, and decompose
\[\widetilde\X_t=\grad\Phi_t\circ\pi_t,\]
where $\pi_t$ is a measure-preserving pushforward mapping, namely
$(\pi_t)_\sharp\, 1=1$. We get
\[(\grad\Phi_t)_\sharp\,1 =
  (\widetilde\X_t)_\sharp\,1=\rho(\cdot,t)\quad\Rightarrow\quad
  \big((\grad\Phi_t)^{-1}\big)_\sharp\,\rho(\cdot,t)=1.\]
Take $\psi_t$ to be the Legendre transformation of
$\Phi_t$, so that $\grad\psi_t=(\grad\Phi_t)^{-1}$.
Then, $\psi_t$ is a solution of the Monge-Amp\`ere equation
\eqref{eq:MApf}.

\begin{proposition}
  Let $(\rho,\u,\phi)$ be a solution of
  \eqref{eqs:EMA}. Assume $\rho$ satisfies
  \eqref{eq:rhononvac} and $\u$ is Lipschitz. Then, $\X_t$ solves the
  following differential equation
  \begin{equation}\label{eq:Xttgen}
      \pa_t^2\X_t(\x)=-\kappa\X_t(\x)+\kappa\,\pi_t\circ\bG^{-1}(\x),\quad
      \X_0(\x)=\x,\quad \pa_t\X_0(\x)=\u_0(\x).
    \end{equation}
    \begin{proof}
      The polar factorization $\X_t\circ\bG=\grad\Phi_t\circ\pi_t$
      implies $\psi_t\circ\X_t=\pi_t\circ\bG^{-1}$. We apply
      \eqref{eq:Xtt} and calculate
      \[\pa_t^2\X_t(\x)=-\kappa\grad\phi(\X_t(\x),t)=-\kappa
      \big(\X_t(\x)-\grad\psi_t(\X_t(\x))\big)=-\kappa
      \big(\X_t(\x)-\pi_t(\bG^{-1}(\x))\big).\]
    \end{proof}
  \end{proposition}

  Unlike the radially symmetric case where $\pi_t(\x)=\x$, the
  measure-preserving mapping $\pi_t$ can vary at different time $t$.
  Therefore, we are not able to obtain an explicit solution of $\X_t$
  from \eqref{eq:Xttgen}. Moreover, it is unclear whether $\pi_t$ is a
  diffeomorphism, or it could lose differentiability at some finite
  time. This has a big impact towards the regularity
  of the solution. We will leave the study of the regularity
  properties of $\pi_t$ in future investigation.

\bibliographystyle{plain}
\bibliography{Eulerian}

\end{document}